\documentclass[11pt]{amsart}
\usepackage{amssymb}
\usepackage{amsmath,amscd}
\usepackage{combelow}
\usepackage{mathtools}
\usepackage{xcolor}
\usepackage{hyperref}
\usepackage{geometry}
\newcommand{\vv}{\vert \vert}

\newcommand{\R}{\mathbb{R}}
\newcommand{\Q}{\mathbb{Q}}
\newcommand{\N}{\mathbb{N}}
\newcommand{\Z}{\mathbb{Z}}
\newcommand{\C}{\mathbb{C}}

\newcommand{\Symp}{\operatorname{Symp}}
\newcommand{\supp}{\operatorname{supp}}

\newcommand{\Ham}{\operatorname{Ham}}
\newcommand{\tHam}{\widetilde{\operatorname{Ham}}}

\newcommand{\Fix}{\operatorname{Fix}}
\newcommand{\Per}{\operatorname{Per}}
\newcommand{\Id}{\operatorname{Id}}

\newcommand{\bH}{\bar{H}}
\newcommand{\tV}{\tilde{\varphi}}

\newcommand{\Spec}{\operatorname{Spec}}

\newcommand{\scon}{\overset{C^{\infty}}{\longrightarrow}}

\newtheorem{theorem}{Theorem}
\newtheorem{corollary}[theorem]{Corollary}

\newtheorem*{question*}{Question}

\newtheorem{lemma}[theorem]{Lemma}
\newtheorem{proposition}[theorem]{Proposition}
\newtheorem{claim}[theorem]{Claim}

\newtheorem*{lemma*}{Lemma}
\newtheorem*{theorem*}{Theorem}
\newtheorem*{remark*}{Remark}
\newtheorem*{definition*}{Definition}
\newtheorem{remark}[theorem]{Remark}
\theoremstyle{remark}

\newtheorem*{remarks*}{Remarks}

\newtheorem*{conjecture*}{Conjecture}

\theoremstyle{definition}

\newtheorem*{claim*}{Claim}

\newtheorem*{example*}{Examples}

\title{The strong closing lemma and Hamiltonian pseudo-rotations}
\author{Erman \c C\. inel\. i and Sobhan Seyfaddini}

\begin{document}
\date{\today}
\maketitle
\begin{abstract}
    We prove the strong $C^\infty$ closing property, as formulated by Irie, for a class of Hamiltonian diffeomorphisms which includes all pseudo-rotations of $\C P^n$ as well as all Anosov-Katok pseudo-rotations.   
\end{abstract}

\tableofcontents

\section{Introduction}
The main goal of this paper is to establish a variant of the $C^\infty$  closing lemma, 
called the strong closing lemma, formulated by Irie \cite{Irie22}, for certain Hamiltonian diffeomorphisms.

\medskip

The $C^\infty$ closing lemma is one of the most significant open problems in dynamical systems.  In the conservative setting, where the system preserves a symplectic or a volume form, it is often formulated as follows: for any open set $U$, there exists a $C^\infty$-small perturbation of the system which has a periodic point/orbit passing through $U$. Although the $C^\infty$ closing lemma does not hold\footnote{The $C^1$ closing lemma does hold in general settings \cite {Pugh1, Pugh2, Pugh-Robinson}.} in full generality \cite{Herman1, Herman2}, it was recently established for Reeb flows in dimension three \cite{Irie-Reeb}, Hamiltonian diffeomorphisms of surfaces \cite{Asaoka-Irie} and more generally for all area-preserving diffeomorphisms of surfaces \cite{CGPZ, HE, CGPPZ}. These results were obtained via purely low dimensional techniques and the higher dimensional case remains largely open. 

Recently, Irie \cite{Irie22} formulated the \textbf{strong closing property} for contact forms: a contact form $\lambda$ is said to satisfy the strong closing property if for every non-zero function $g\geq 0$ (or $g \leq 0)$ there exists $s\in[0, 1]$ such that the  contact form $(1+sg) \lambda$ has a closed Reeb orbit passing through the support of $g$.  Irie moreover conjectured that this property is satisfied by the Reeb flow on the boundary of an ellipsoid; this  was recently proven by Chaidez, Datta, Prasad and Tanny \cite{CDPT} using contact homology.  Shortly thereafter similar results were obtained by Xue using the techniques of KAM theory \cite{Xue}.  

\subsection{Main Result}
We now introduce the class of Hamiltonian diffeomorphisms for which we establish the strong closing property. Let $(M, \omega)$ be a closed and connected symplectic manifold. Its group of Hamiltonian diffeomorphisms $\Ham(M, \omega)$ admits a remarkable  (bi-invariant) distance called the spectral metric; it is usually denoted by $\gamma$ and it was introduced in the works of  Viterbo, Schwarz and Oh \cite{viterbo92, schwarz, oh}. We say $\psi \in \Ham$ is  \textbf{$\gamma$-rigid}\footnote{Rigid maps are also known as approximate identities; see \cite{Ginzburg-Gurel19}.} if there exists a sequence of natural numbers $k_i \to \infty$ such that $\psi^{k_i} \to \Id$ in the $\gamma$ distance. 

Next we state our main result. Note that we denote by $\varphi_G$ the time-one map of the flow of a Hamiltonian $G$; see Section \ref{sec:hamiltonians} for notations and definitions.

\begin{theorem}
\label{thm:closing}
Every $\gamma$-rigid  Hamiltonian diffeomorphism $\psi$ satisfies the strong closing property: for every non-zero Hamiltonian  $G \ge 0 $ (or $G\le 0$) supported away from periodic points of $\psi$ the composition $\varphi_{G}\psi$  has a periodic point passing through the support of $G$. 
\end{theorem}

In thinking about the above theorem, it is helpful to keep in mind that the ultimate goal here is to show that there exists {\bf some} $C^\infty$ small perturbation which creates a periodic point in a prescribed location.  According to the theorem, for $\gamma$-rigid maps, one can achieve this goal  via {\bf any} local perturbation by a positive  Hamiltonian.

\begin{remark}\label{rmk:support-hypothesis}
The conclusion of Theorem \ref{thm:closing} does not hold without the assumption that $G\ge 0$ (or $G\le 0)$.  We explain this below in Section \ref{sec:misc}.
\end{remark}

\subsection{Examples of $\gamma$-rigid maps: rotations and pseudo-rotations.}

Suppose that $(M, \omega)$ admits a smooth Hamiltonian torus action $R$, e.g.\ $\C P^n$, products of $\C P^n$, and more generally toric symplectic manifolds.  We will refer to any element of the torus action as a \textit{rotation}. Any rotation is $\gamma$-rigid.   Indeed, the iterates of a rotation $C^\infty$ accumulate at the $\Id$ and this implies $\gamma$-rigidity (as $\gamma$ is known to be continuous in $C^\infty$).

\emph{Pseudo-rotations} yield more interesting examples of $\gamma$-rigid maps.  Recall that a Hamiltonian diffeomorphism $\varphi \in \Ham(M, \omega)$ is said to be a  (Hamiltonian) pseudo-rotation if it has finitely many periodic points.\footnote{There exist several working definitions of Hamiltonian pseudo-rotations in the literature; see  \cite[Def.\ 1.1]{CGG19b} and the discussion therein.} Such diffeomorphisms have been of great interest in  dynamical systems and symplectic topology; see, for example, \cite{Anosov-Katok, Fathi-Herman, Fayad-Katok, BCL04, BCL06, Bramham15a, Bramham15b, LeCalvez16, AFLXZ, Ginzburg-Gurel18a,  Ginzburg-Gurel18b, Ginzburg-Gurel19, CGG19a, CGG19b, Shelukhin19}.   Below we discuss some classes of pseudo-rotations for which $\gamma$-rigidity has been established.

\medskip
\subsection*{Anosov-Katok pseudo-rotations.}   Suppose that $(M, \omega)$ carries a Hamiltonian $S^1$-action $R$ whose fixed point set $\Fix(R)$ is finite\footnote{In practice, to construct interesting examples of AKPRs one needs the action to satisfy additional requirements: in addition to $\Fix(R)$ being finite, the action must also be locally free in $M\setminus \Fix(R)$.  Moreover, the conjugating symplectomorphisms $h_k$ are taken to coincide with the identity in a neighborhood $U_k$ of $\Fix(R)$.}, e.g. $\C P^n$, products of $\C P^n$, and more generally toric symplectic manifolds.  We call a pseudo-rotation $\varphi$ an \emph{Anosov-Katok pseudo-rotation} (abbreviated \emph{AKPR}) if there exist sequences $h_k \in \Symp (M,\omega)$ and $\alpha_k \in \Q / \Z$ such that $h_k^{-1} R_{\alpha_k} h_k \scon  \varphi.$

\medskip
 Every AKPR is $\gamma$-rigid; see \cite[Corol.\ 15]{Jok-Sey}.  
 \begin{corollary}
\label{corol:ak}
Every Anosov-Katok pseudo-rotation satisfies the strong closing property.
\end{corollary}
 
 This yields a rich source of maps with the strong closing property  which we will now examine more closely.\footnote{Observe that in the case of AKPRs the $C^\infty$ closing lemma holds for trivial reasons because every AKPR is the $C^\infty$ limit of finite order maps.}   The simplest AKPRs are the irrational rotations, that is $R_\alpha$  with $\alpha$ an irrational vector. As far as we know, Theorem \ref{thm:closing} is non-trivial even in this simple setting.
 However, in contrast to the setting considered in \cite{CDPT, Xue} where the dynamical system is integrable,  an AKPR could have complicated dynamics. For example, every $(M, \omega)$ with a circle action as above admits an AKPR with a dense orbit.  Moreover, $\C P^n$, products of $\C P^n$, and more generally all toric symplectic manifolds, admit AKPRs with a finite number of ergodic measures; this fact was proven for $\C P^1$ by Fayad-Katok \cite{Fayad-Katok} and was generalized to higher dimensions by Le Roux and the second author in \cite{LeRoux-Sey}. Both constructions rely on the conjugation method of Anosov and Katok \cite{Anosov-Katok}.

\subsection*{Pseudo-rotations of $\C P^n$.} Consider $\C P^n$ equipped with the standard Fubini-Study symplecitc structure.  It is proven in \cite{Ginzburg-Gurel18a} that every pseudo-rotation of $\C P^n$ is $\gamma$-rigid\footnote{In \cite{Ginzburg-Gurel18a}, this is proven for pseudo-rotations of $\C P^n$ with exactly $(n+1)$ periodic points. By Franks' theorem every pseudo-rotation of $\C P^1$ has exactly 2 periodic points. Hypothetically, in higher dimensions, a pseudo-rotation of $\C P^n$ can have more than $n+1$ periodic points (conjecturally, such pseudo-rotations do not exist). However,  in view of the recent results by Shelukhin \cite{Shelukhin22}, the proof of $\gamma$-rigidity given in \cite{Ginzburg-Gurel18a} extends to this hypothetical case as well.}  and so we conclude that every pseudo-rotation of $\C P^n$ satisfies the strong closing property.

The simplest examples of pseudo-rotations on $\C P^n$ are the irrational rotations (see \cite{Ginzburg-Gurel18b} for a detailed study) which have very simple dynamics and, as we explain below, are intimately connected to the dynamics of Reeb flows on ellipsoids.  As alluded to above, a general pseudo-rotation of $\C P^n$ could be far more complicated and, for example, could have exactly $(n+2)$ ergodic measures; see \cite{Fayad-Katok} for $n=1$ and \cite{LeRoux-Sey} for general $n$. We remark that the examples of \cite{Fayad-Katok, LeRoux-Sey} are AKPRs.

\subsection{Miscellaneous remarks and observations.}\label{sec:misc}
We collect here some remarks and observations about various aspects of the above results.  

\subsection*{Reeb flows on ellipsoids and  rotations of $\C P^n$.}
Our interest in the strong closing property for pseudo-rotations was triggered by the recent work of Chaidez, Datta, Prasad and Tanny \cite{CDPT} proving this property for Reeb flows on ellipsoids. There are striking similarities between the dynamics of  ellipsoids and rotations of $\C P^n$  which allow one to relate the closing property in one setting to the other.  

Indeed, it is not difficult to see that for every ellipsoid $E \subset \C^n$, there exists a natural quadratic Hamiltonian $Q: \C P^n \to \R$ such that, up to rescaling, $E$ can be realized as a level set of $Q$.  Moreover, a (1-parameter) perturbation of the contact form on $E$ corresponds to a perturbation of $Q$.  This observation leads to a proof of the strong closing property for ellipsoids using the same property for rotations of $\C P^n$; we will give a brief outline of this in Section \ref{sec:ellipsoids}. 

Conversely, we expect that the strong closing property for ellipsoids can be used to prove the closing property for rotations of $\C P^n$.

\subsection*{$\gamma$-rigid maps.} 
    If $\varphi_i \in \Ham(M_i, \omega_i)$, where $i=1,2$,  are $\gamma$-rigid then $\varphi_1\times \varphi_2 \in \Ham(M_1 \times M_2, \omega_1 \oplus \omega_2)$ is also $\gamma$-rigid.  We do not know of examples of $\gamma$-rigid maps beyond those mentioned in the previous section (i.e.\ rotations, AKPRs and pseudo-rotations of $\C P^n$) and their products.

We remark here that one could obtain a (theoretically) larger collection of $\gamma$-rigid maps by relaxing the $C^\infty$ convergence in the definition of AKPRs to $C^1$ convergence (while still supposing that the limit is a $C^\infty$ Hamiltonian diffeomorphism).  Indeed, the argument given in \cite{Jok-Sey} adapts easily to yield $\gamma$-rigidity for such maps because $\gamma$ is known to be $C^1$ continuous.  In the case of $\C P^n$, one can even further relax the $C^\infty$ convergence to $C^0$ convergence because $\gamma$ is $C^0$ continuous on $\C P^n$ \cite{Shelukhin-18}; we further comment on connections to the $C^0$ topology below.

As remarked in \cite{Ginzburg-Gurel19, CGG22}, one expects that few manifolds admit $\gamma$-rigid maps. For instance, it is conjectured that symplectically aspherical manifolds (e.g. positive genus surfaces) do not admit such maps. Moreover, in the case of manifolds that do admit them, $\gamma$-rigidity is expected to be a $C^\infty$-meager property. It was recently shown in \cite[Cor.\ 2.13]{CGG22} that a $C^\infty$-generic Hamiltonian diffeomorphism of a closed surface is not $\gamma$-rigid.

\subsection*{$C^0$-rigid maps.} $\gamma$-rigidity is intimately related to the notion of $C^0$-rigidity which has been studied in dynamical systems; see for example \cite{Kolev-Peroueme}.  A homeomorphism $\psi$ is called $C^0$-rigid (or sometimes recurrent) if there exists a sequence of natural numbers $k_i \to \infty$ such that $\psi^{k_i} \to \Id$ in the $C^0$ topology. 

It is conjectured that the $\gamma$-distance is continuous in the $C^0$ topology and so we expect that $C^0$-rigidity implies $\gamma$-rigidity. The $C^0$-continuity of $\gamma$  has been an important theme in $C^0$ symplectic topology and has thus far been established for $\R^{2n}$ \cite{viterbo92}, surfaces \cite{SeyC0}, aspherical \cite{BHS} and negative monotone \cite{Kawamoto} symplectic manifolds, and $\C P^n$ \cite{Shelukhin-18}.

 \subsection*{On the hypotheses of Theorem \ref{thm:closing}}
 As promised in Remark \ref{rmk:support-hypothesis}, we present an example here demonstrating that the conclusion of  Theorem \ref{thm:closing} does not hold without the assumption that $G\ge 0$ (or $G\le 0)$. 
 
Let $\psi \in \Ham$ be a pseudo-rotation and $U \subset M$ be an open set such that $\psi(U) \cap U =\emptyset$; thus, $U$ is contained in the complement of $\Fix(\psi)$, the fixed point set of $\psi$. Pick an autonomous Hamiltonian $F$ supported in $U$ and consider the commutator  $\varphi_F \psi \varphi_F^{-1} \psi^{-1}$; it can be written as the time-1 map of the autonomous Hamiltonian $G: = F - F\circ \psi^{-1}$ which is supported in  $U \cup \psi (U)$ (which is in the complement of $\Fix(\psi)$).  Now, note that $\varphi^s_{G}\psi = \varphi^s_{F} \psi \varphi_{F}^{-s}$ which is conjugate to $\psi$, and hence a pseudo-rotation.  Moreover, $\Fix(\varphi^s_{G}\psi)=\Fix(\psi)$, so the conclusion of Theorem \ref{thm:closing} does not hold for $G$.

\subsection*{The closing lemma and generating function theory}

  On general symplectic manifolds the construction of the spectral metric $\gamma$, and hence the proof of the strong closing lemma for $\gamma$ rigid maps, requires the theory of Hamiltonian Floer homology.  In the case of $\C P^n$, however, we expect that Floer theory can be replaced with the finite-dimensional methods of generating function theory which are considered to be simpler.  Indeed,  the spectral metric on $\C P^n$  can be constructed via generating functions (see for example \cite{Allais}) and we believe, moreover, that all the properties of the spectral metric which are necessary for proving Theorem \ref{thm:closing} can also be established via generating functions.
  
  Note that, in view of the above discussion on Reeb flows on ellipsoids, this also yields a generating functions proof of the closing property in that setting.
  
  \subsection*{The closing lemma and Gromov-Witten theory}  Ginzburg and G\"urel's proof \cite{Ginzburg-Gurel18a} of $\gamma$-rigidity for all pseudo-rotations of $\C P^n$ relies heavily on the structure of the quantum product and Gromov-Witten invariants of $\C P^n$.  In the words of Fish and Hofer \cite{Fish-Hofer20}, this provides further evidence for the general belief that ``the validity
of the Hamiltonian $C^\infty$-closing lemma is intimately connected to the existence of a sufficiently rich Gromov-Witten theory of the ambient space.''  We refer the reader to \cite{Fish-Hofer20} for interesting speculations on the closing lemma and its potential connections to the the feral curves theory from \cite{Fish-Hofer-feral}.

 \subsection*{Organization of the paper}
 In Section \ref{sec:prelim}, we introduce some of our notation and recall some basic notions from symplecitc geometry.  In Sections \ref{sec:floer} and \ref{sec:spec}, we review the aspects of the Hamiltonian Floer theory and spectral invariants which are needed for our arguments.  Section \ref{sec:pf} contains the proof of Theorem \ref{thm:closing}.  Finally, in Section \ref{sec:ellipsoids} we sketch a proof of the strong closing property for ellipsoids which is based on Theorem \ref{thm:closing}.

\subsection*{Acknowledgements} We thank Dustin Connery-Grigg, Dan Cristofaro-Gardiner, Sylvain Crovisier, Viktor Ginzburg, Ba\c{s}ak G\"urel, Vincent Humili\`ere, Kei Irie, Patrice Le Calvez, Frédéric Le Roux, Cheuk Yu Mak, Rohil Prasad and  Shira Tanny for helpful comments and conversations. We are particularly grateful to Shaoyun Bai for discussions on the VFC package of \cite{Pardon}.

Both authors are supported by the ERC Starting Grant number 851701.

\section{Preliminaries}
\label{sec:prelim}
In this section we introduce some of our notations and conventions and recall some basic notions from symplectic geometry.

\subsection{Notations and conventions}\label{sec:conventions}
Let $(M^{2n}, \omega)$ be a closed and connected symplectic manifold.  We denote by $C^\infty ( \R/ \Z \times M)$ the set of smooth one-periodic Hamiltonians on $M$. The support $\supp(H) \subset M$ of a Hamiltonian $H \in C^\infty ( \R/ \Z \times M)$ is defined by
$$
\supp(H) := \overline{ \{ x \in M \ \vert \ H(x,t) \neq  0 \ \text{for some} \ t \in \R/\Z\}}.
$$

Given $H \colon \R/ \Z \times M \to \R$, the Hamiltonian vector field $X_H$ of $H$ is given by 
\[
\omega(X_H, \cdot) :=-dH.
\]
The vector field $X_H$ generates a time-dependent flow (or isotopy) $\varphi_H^t$ called the Hamiltonian flow of $H$; we denote its time-one map by $\varphi_H$. The set of all such time-one maps forms the group of Hamiltonian diffeomorphisms $\Ham(M,\omega)$ of $(M, \omega)$. It is a subgroup of the group of symplectomorphisms $\Symp(M, \omega)$ of $(M, \omega)$.

Let $x \colon \R/ \Z \to M$ be a smooth contractible loop in $M$. 
A capping of $x$ is a map $v \colon D^2 \to M$ such that $v\vert_{\partial D^2}=x$. We refer to the pair $(x,v)$ as a {\it capped loop} and say that two capped loops $(x_1, v_1), (x_2, v_2)$ are equivalent, if $x_1 = x_2$ and the sphere $v_1 \sharp (-v_2) \in \ker (\omega) \cap \ker (c_1)$ where $c_1$ is the first Chern class of $TM$. From here on-wards, we only consider capped loops up to this equivalence relation.

Let $\Omega$ be the space consisting of capped loops in $M$.  
Given a Hamiltonian $H$, the associated action functional is a mapping $\mathcal{A}_H: \Omega \rightarrow \R$ defined by
\[
\mathcal{A}_H (x,v) := -\int_v \omega + \int_0^1 H(t, x(t)) dt.
\]
The critical points of $\mathcal{A}_H$ consist of capped loops $(x, v)$ such that $x$ is a contractible  one-periodic orbit of the Hamiltonian vector field $X_H$. We denote this set by $\Omega_H \subset \Omega$. The set of critical values of $\mathcal{A}_H$
$$
\Spec(H):= \{\mathcal{A}_H (x,v)\ \vert \  (x,v) \in \Omega_H \}
$$ 
is called the action spectrum of $H$. 

A one-periodic orbit $x$ of $X_H$ is called non-degenerate if the linearized return map   $$D\varphi_H \colon T_{x(0)}M \to T_{x(0)}M$$  has no eigenvalues equal to one. We say that the Hamiltonian $H$ is non-degenerate if all one-periodic orbits of $X_H$ are non-degenerate. The Conley-Zehnder index $\mu(x,u)$ of a capped non-degenerate orbit $(x,u) \in \Omega_H$ is defined as in \cite{Sa-lecture}. In this paper, $\mu$ is normalized so that $\mu(x, u)=n$ when $x$ is a maximum, with constant capping, of an autonomous  Hamiltonian with small Hessian.

We denote by $\Fix(\varphi)$ the set of fixed points of $\varphi$ and by $\Per(\varphi)$ the set of all periodic points, i.e., 
\[
\Per(\varphi) := \cup_{k \in \N} \Fix(\varphi^k). 
\]
We write $\Fix_c (\varphi) \subset \Fix(\varphi)$ and $\Per_c(\varphi) \subset \Per(\varphi)$ for the subsets formed by the contractible periodic orbits of a lift $\tV \in \tHam (M,\omega)$ of $\varphi$. A fixed point being contractible does not depend on the lift; this fact is a consequence of the Arnold conjecture, see \cite{schwarz}, for example.

\subsection{Remarks on Hamiltonians and related operations}\label{sec:hamiltonians}
Recall that the universal cover of $\Ham(M,\omega)$, which we denote by $  \tHam(M, \omega)$ consists of the set of all Hamiltonian isotopies $\varphi^t_H, t\in [0,1]$ considered up to homotopy relative endpoints.   We denote by $\tV_H \in \tHam(M, \omega)$ the natural lift of $\varphi_H$ (given by the isotopy $\varphi_H^t$)  to the universal cover $\tHam (M, \omega)$.

Up to reparametrization, every Hamiltonian isotopy can be obtained as the flow of a Hamiltonian $H$ which vanishes for values of $t$ close to $0, 1$.  Note that reparametrization does not affect the lift of an isotopy to $\tHam (M, \omega)$.  We will assume throughout the paper that all Hamiltonians vanish near $t=0,1$; this will be helpful in defining the concatenation operation below. 

The concatenation $H_1 \sharp \cdots \sharp H_k$ of the Hamiltonians $H_1, \ldots, H_k$ is given by 
\begin{equation*}\label{eq:sharp}
H_1 \sharp \cdots \sharp H_k(t, x) := kH_i(k (t -(i-1)/k), x)
\end{equation*}
for $t \in [(i-1)/k, i/k]$. Note that $H_i(k (t -(i-1)/k), x)=H_i(k t, x)$ since our Hamiltonians are one-periodic in time.  The Hamiltonian isotopy generated by $H_1 \sharp \cdots \sharp H_k(t, x)$ is simply (a reparametrization of) the concatenation of the Hamiltonian flows of $H_1, \ldots, H_k$. In particular, $\varphi_{H_1 \sharp \cdots \sharp H_k } = \varphi_{H_1} \circ \cdots \circ \varphi_{H_k}$.  Moreover, the two Hamiltonian isotopies  $\varphi^t_{H_1 \sharp \cdots \sharp H_k }$ and  $\varphi^t_{H_1} \circ \cdots \circ \varphi^t_{H_k}$ are homotopic rel endpoints, i.e. 
$$ \tV_{H_1 \sharp \cdots \sharp H_k } = \tV_{H_1} \circ \cdots \circ \tV_{H_k}.$$   

The inverse $\bH$ of the Hamiltonian $H$ is defined by 
\[
\bH(t,x):=-H(t, \varphi^t_H(x))
\]
whose flow is given by $(\varphi^t_H)^{-1}$ and so  $\tV_{\bH}=\tV_H^{-1}$.  

Given $\psi \in \Symp(M, \omega)$, define the pull back Hamiltonian $H\circ \psi$ by 
$$H\circ \psi (t, x) := H(t, \psi(x)).$$  The corresponding Hamiltonian flow is given by
$\varphi^t_{H\circ \psi} := \psi^{-1} \varphi^t_{H} \psi$.  Now, consider the case where $\psi = \varphi_K$; it is not difficult to check that 
\begin{equation*}
\label{eq:conjugation}
  \tV_{H \circ \varphi_K} =  \tV_{\bar{K}} \, \tV_{H} \,  \tV_K.
\end{equation*}

\medskip

We say that the Hamiltonian $H$ is mean normalized if $\int_M H_t \omega^n =0 $ for all $t\in [0,1]$. Let $\Delta(H)$ be the average of $H$:
\begin{equation}\label{eq:Delta}
    \Delta(H):= \int_0^1 \int_M H_t \omega^n dt.
\end{equation}

Observe that $\Delta$ satisfies the following properties 
\begin{itemize}
    \item  $\Delta(H_1 \sharp \cdots \sharp H_k)=\sum \Delta(H_i).$
    \item  $\Delta (H\circ \psi) = \Delta(H)$.
    \item $\Delta(\bar{H}) = - \Delta(H)$.
\end{itemize}

\begin{remark}
\label{rk:conca}
We record here the following observation which will be used below. Introducing a Hamiltonian of the form $K \sharp \bar{K}$ anywhere in $H_1 \sharp \cdots \sharp H_k$, that is, say, for instance
\[
H_1 \sharp \cdots K \sharp \bar{K} \cdots \sharp H_k,
\]
does not change $\tV_{H_1 \sharp \cdots \sharp H_k}$ or $\Delta(H_1 \sharp \cdots \sharp H_k)$. 
\end{remark}

\section{Floer theory}
\label{sec:floer}
In this section we briefly review some aspects of Hamiltonian Floer theory.  To prove  Theorem \ref{thm:closing} in full generality, we need the construction of Hamiltonian Floer theory for arbitrary closed symplectic manifolds which requires the use of virtual techniques such as the polyfold theory of Hofer-Wysocki-Zehnder \cite{HWZ17}  or Pardon's VFC package \cite{Pardon}; here, we will use \cite{Pardon}.  An important feature of the VFC package, on which we will be relying on, is that, by \cite[]{Pardon},   the virtual count and the geometric count coincide when transversality holds (see also \cite[Prop.\ 4.33]{Par19}).  Finally, we remark that the use of virtual techniques could be avoided if we were to limit Theorem \ref{thm:closing} to, for example,  monotone symplectic manifolds where one can rely on the classical transversality techniques \cite{FHS}.

Let $H \colon \R/ \Z \times M \to \R$ be a non-degenerate Hamiltonian and $J$ be an $\omega$-compatible almost complex structure on $TM$. As a vector space the Floer chain-complex is given by 
$$
CF(H,J) := \{\sum_{(x,v) \in \Omega_H} a_{(x,v)} (x,v) \ \vert \ a_{(x,v)} \in \Q;   \ \forall C \in \R, \ \vert\{ a_{(x,v)} \neq 0, \ \mathcal{A}_H(x,v) >C\} \vert < \infty \},
$$
and the differential $\partial_{(H,J)}$ is defined as follows. Let $(x,v)$, $(y,w) \in \Omega_H$ with 
$$ \mu(x,v)-\mu(y,w) =1$$
and consider  solutions $u \colon \R \times S^1 \to M$ to the Floer equation
\begin{equation}
\label{eq:fe}
    \partial_s u(s,t) + J(u(s,t)) (\partial_t u(s,t) - X_{H_t} (u(s,t)))=0
\end{equation}
with asymptotics 
$$
    \lim_{s\to -\infty} u(s,t) = x(t), \  \lim_{s\to \infty} u(s,t) = y(t)
$$
and such that $(y, v\sharp u) = (y, w) \in \Omega_H$.  Note that $\mathbb R$ acts on the set of all $u$ as above by translation: $a\cdot u(s,t) \mapsto u(a +s, t).$

The differential $ \partial_{H,J} : CF(H,J) \rightarrow CF(H,J)$ is defined via the equation $$\partial_{(H,J)} (x, v) = \sum_{(y, w)} a_{(y,w)} (y,w), $$ where the coefficient $a_{(y,w)} \in \mathbb Q$ is given by a certain \emph{virtual count} of  the solutions of \eqref{eq:fe} as above, modulo $\R$ translation; see \cite{Pardon}.  The differential satisfies $\partial^2_{(H,J)}=0$ (see \cite{Pardon}).

As we will now explain, the homology $HF(H,J)$ of the complex $(CF(H,J), \partial_{(H,J)})$ does not depend on $H$ or $J$. Indeed, let $(H^0, J^0)$, $(H^1, J^1)$ be two pairs as above. Consider, for instance, an interpolating homotopy
\begin{align*}
    H^s:= (1-\beta(s))H^0 + \beta(s)H^1
\end{align*}  
between the Hamiltonians $H^i$. Here $\beta \colon \R \to [0,1]$ is a smooth monotone increasing function with $\beta(s)\equiv 0$ for sufficiently small and $\beta(s)\equiv 1$ for sufficient large $s \in\R$. Similarly let $J^s$, $ s\in \R$, be a family of almost complex structures which is eventually constant and equal to $J^i$ for sufficiently small and large $s \in \R$. Similarly to the differential, the continuation map
$$
\Psi_{(H_s, J_s)} \colon CF(H_0, J_0) \to CF(H_1, J_1),
$$
is given by the \emph{virtual count}, as defined in \cite{Pardon}, of the compactified moduli space 
$$
\overline{\mathcal{M}((x,v), (y,w))}
$$ 
consisting of solutions to $u \colon \R \times S^1 \to M$ to the parameterized Floer equation 
\begin{equation}
\label{eq:fe_cont}
    \partial_s u(s,t) + J^s(u(s,t)) (\partial_t u(s,t) - X_{H^s_{t}} (u(s,t)))=0
\end{equation}
with asymptotics $(x^i,v^i) \in \Omega_{H^i}$ that satisfy $\mu((x^0,v^0))=\mu((x^1,v^1))$ together with $(x^1, v^0\sharp u)= (x^1, v^1) \in \Omega_{H^1}$. 

The resulting map $\Psi_{(H^s, J^s)}$ is a chain-map and it induces an isomorphism in homology (see \cite{Pardon}).

\subsection{Energy estimates.} \label{sec:energy}

Below we discuss some energy estimates which will be important to us later in the proofs. The energy $E(u)$ of a solution $u$ to either of the Floer equations \eqref{eq:fe}, \eqref{eq:fe_cont} is given by   $$ E(u) = \int_{-\infty}^{\infty}\int_{S^1} \vert \partial_s u \vert^2 dt ds.$$ 
 In the case of the differential \eqref{eq:fe}, a direct computation shows that 
$$
E(u)=\mathcal{A}_H(x, v)-\mathcal{A}_H(y, w)
$$
for $u \in \mathcal{M}((x,v), (y,w))$ (see, for instance, \cite[Sec. 4.3]{BPS}). It follows that the solution space $\mathcal{M}((x,v), (y,w))$ is empty if $\mathcal{A}_H(x, v) -\mathcal{A}_H(y, w)<0$.  
Let us define $\mathcal{A}_H \colon CF(H, J) \to \R$ by
$$
\mathcal{A}_H (\sum a_{(x,v)} (x, v) := \max \{ \mathcal{A}_H(x,v) \ \vert \ a_{(x,v)} \neq 0\}
$$
where we set  $\mathcal{A}_H(0) = - \infty$. Observe that we have $\mathcal{A}_H(\partial_{(H,J)} (c)) \leq \mathcal{A}_H(c)$ for all $c \in CF(H,J)$. In other words the complex $(CF(H,J), \partial_{(H,J)})$ is filtered by the action $\mathcal{A}_H$. 

\medskip

Next, we consider energy estimates for continuation maps. This time a direct computation gives 
$$
E(u)= \mathcal{A}_{H^0} (x, v) - \mathcal{A}_{H^1}(y, w)  + \int_{-\infty}^{\infty} \int_{S^1} \partial_s H^s (u(s,t)) dt ds
$$
(see \cite[Sec.\ 4.4]{BPS}).   For our proofs, we will be interested in the case where $H^0 \geq H^1$ and $H^s$ is an interpolating homotopy between $H^0, H^1$.  In this case, since $\partial_s H^s \leq 0$, if the compactified moduli space $\overline{\mathcal{M}((x,v), (y,w))} \neq \emptyset $ then 
\begin{equation}\label{eq:action_decreseas}
    \mathcal{A}_{H^1}(y, w) \leq \mathcal{A}_{H^0}(x,v).
\end{equation}
Moreover, the equality can only hold if $(x,v) = (y,w)$ and $ \overline{\mathcal{M}((x,v), (x,v))}$ consists of the constant (in $s$) solution to \eqref{eq:fe_cont}.

\section{Spectral invariants}
\label{sec:spec}

In this section we recall the definition and some relevant properties of spectral invariants following Oh \cite{oh} and Usher \cite{usher}. These invariants were first introduced by Viterbo \cite{viterbo92} in the Euclidean setting, by Schwarz \cite{schwarz} in the setting of aspherical symplectic manifolds (i.e. $\pi_2(M)=0$), and by Oh \cite{oh} on general symplectic manifolds.

Let $\Gamma:= \pi_2(M) / \ker (\omega) \cap \ker (c_1)$ and 
$$
\Lambda_{\Gamma} := \{ \sum_{A \in \Gamma} b_A A \ \vert \ b_{A} \in \Q;   \ \forall C \in \R, \ \vert\{ b_{A} \neq 0, \ \omega(A) <C\} \vert < \infty \}
$$
be the associated Novikov field. For any pair $(H,J)$ as in Section \ref{sec:floer}, the spectral invariant $c(H,J)$ of $H$ is defined by 
$$
c(H) := \inf \{ \mathcal{A}_H (c) \ \vert \ [c] = PSS([M]) \}
$$
where 
\begin{equation}
    \label{eq:PSS}
PSS \colon H_* (M; \Q) \otimes_{\Q} \Lambda_{\Gamma} \to HF (H,J)
\end{equation}
 is the PSS-isomorphism \cite{PSS} and $[M] \in H_*(M)$ is the fundamental class. Roughly speaking, the spectral invariant $c(H)$ of a Hamiltonian $H$ is the minimal action value at which the fundamental class $[M]$ appears in the filtered Hamiltonian Floer homology of $H$. The invariant $c(H,J)$ does not depend on  $J$. Moreover, by the Lipschitz continuity \eqref{eq:cont} below, it extends to all, possibly degenerate, Hamiltonians. We note that the PSS-isomorphism commutes with continuation maps (see Section \ref{sec:floer}).

\begin{remark}
Observe that the Novikov field $\Lambda_{\Gamma}$ acts on $CF(H,J)$ by re-capping. In particular, one can see $CF(H,J)$ as a finite dimensional vector space over $\Lambda_{\Gamma}$ with $\dim_{\Lambda_{\Gamma}} CF(H,J) = \vert \Fix_c (\varphi_H) \vert$. Hence, by \eqref{eq:PSS}, the number of contractible one periodic orbits of $X_H$ is bounded from below by $\dim H_*(M; \Q)$.  
\end{remark}

We now recall some of the properties of the spectral invariant 
$$
c \colon C^\infty(\R/ \Z  \times M) \rightarrow \R,
$$
which we will be used below.  Recall the notation $H\sharp G, \bar{H}, H\circ \psi, \Delta(H), \tV$ from Section \ref{sec:hamiltonians}

 \medskip

 \noindent \textbf{Normalization:} Let $H \equiv 0$ be the zero Hamiltonian. We have  
 \begin{equation}
 \label{eq:normalization}
     c \big(H \big)=0.
 \end{equation}

\medskip

\noindent \textbf{Shift:} Let $a: \R/ \Z  \rightarrow \R$ be smooth.  Then, 
\begin{equation}
\label{eq:shift}
    c \big(H + a \big) = c\big(H  \big) + \int_{0}^1 a(t) \, dt.
\end{equation}

\medskip

\noindent \textbf{Symplectic invariance:} For any $\psi \in \Symp(M, \omega)$ and $H \in C^\infty(\R/\Z \times M)$ we have 
\begin{equation}
\label{eq:symp_inv}
    c \big(H \circ \psi ) = c\big(H  \big).
\end{equation}

\medskip

\noindent \textbf{Homotopy invariance:} Suppose that  $\tV_H=\tV_G$ and that  $\Delta(G)=\Delta(H)$.  Then,  
\begin{equation}
\label{eq:normal}
    c(H)= c\big(G\big).
\end{equation}

\medskip

\noindent \textbf{Monotonicity:}  If $G \leq H$, then  
\begin{equation}
\label{eq:monotone}
    c(G) \leq c(H).
\end{equation}

\medskip

\noindent \textbf{Lipschitz continuity:}  Let $\vv \cdot \vv_\infty$ denote the sup norm on $ C^\infty( \R/ \Z \times M)$.  We have
\begin{equation}
\label{eq:cont}
    \vert c(G) - c(H) \vert \leq  \vv G - H \vv_\infty  
\end{equation}

\medskip

\noindent \textbf{Triangle inequality:}   We have
\begin{equation}
\label{eq:tri}
    c( H\sharp G) \leq c (H) + c (G).
\end{equation}

The proofs of all of the above properties, except for Homotopy invariance which is usually stated differently, can be found in \cite{oh, McDuff-Salamon-Jcurves}.  Usually, for example in the aforementioned references, Homotopy invariance is stated for mean-normalized Hamiltonians: if $\tV_G = \tV_H$ and $G, H$ are both mean normalized, then $c(G) = c(H)$.  We leave it to the reader to check that the mean-normalized version of Homotopy invariance combined with the Shift property above implies our version of Homotopy invariance.  

The spectral norm $\gamma(H)$ of a Hamiltonian $H \colon \R/ \Z \times M \to \R$ is defined by 
\[
\gamma(H) :=c(H) + c(\bH)
\]
and the spectral norm $\gamma(\varphi)$ of $\varphi \in \Ham(M,\omega)$ is given by
\begin{equation}\label{eq:spec-norm}
    \gamma(\varphi) := \inf_{\varphi_H=\varphi} \gamma(H).
\end{equation}

The map  $$\gamma: \Ham(M, \omega) \rightarrow \R$$ is a conjugation invariant norm, i.e.\ it satisfies the following properties
\begin{enumerate}
    \item \label{item:non-deg} $\gamma(\varphi) \geq 0$ with equality iff $\varphi = \Id$.
    \item  $ \gamma(\varphi \psi) \leq \gamma(\varphi) + \gamma(\psi).$ 
    \item  $\gamma(\varphi \psi \varphi^{-1}) = \gamma(\psi)$.
\end{enumerate}

Consequently, induces a genuine bi-invariant metric $d_{\gamma}$ on $\Ham(M,\omega)$ by setting $$ d_{\gamma}(\varphi, \psi) :=\gamma(\varphi \psi^{-1}). $$  Recall that bi-invariance means $d_{\gamma}(\theta \varphi, \theta \psi) = d_{\gamma}( \varphi,  \psi) =  d_{\gamma}( \varphi \theta ,  \psi \theta)$.

\section{Proof of Theorem \ref{thm:closing}}
\label{sec:pf}

In this section we present the proof of Theorem \ref{thm:closing}. We begin with some preliminary lemmas.  

\begin{lemma}
\label{lem:spec}
    Let $\psi \in \Ham(M, \omega)$ and $G \colon \R/ \Z \times M \to \R$ be a non-negative Hamiltonian with $\supp (G) \subset M \setminus \Fix_c(\psi)$. Suppose that $\Fix_c(\varphi_{G} \psi)=\Fix_c(\psi)$. Then, for every Hamiltonian $F$ with $\varphi_F= \psi$, we have $c( G \sharp F )=c(F)$.
\end{lemma}

We remark that our proof below relies on an ``upper triangular" argument which has been used in the past in various contexts; see, for example, \cite{Fil-Wehr, Bai-Xu}.

\begin{proof}
    We first handle the non-degenerate case. Suppose that $\psi$ is non-degenerate and fix $F$ with $\varphi_F=\psi$. By Homotopy invariance \eqref{eq:normal} it suffices to show that $c( G \sharp F )=c(0\sharp F)$. 
    
    Let 
    $$
    \Psi \colon CF(G \sharp F) \to CF(0 \sharp F)
    $$
    be a continuation map induced by an interpolating homotopy, as reviewed in Section \ref{sec:floer}. The underlying almost complex structures will be immaterial for the proof, hence we have dropped them from the notation.

    \begin{claim}\label{cl:acion_equal}
    For all $c \in CF(G \sharp F)$, we have 
    \begin{equation} \label{eq:action_equal}
      \mathcal{A}_{0\sharp F}(\Psi(c)) =\mathcal{A}_{G\sharp F}(c).
    \end{equation}
    \end{claim}
    \begin{proof}[Proof of Claim \ref{cl:acion_equal}]
        Let $(x,v)$ denote a capped one-periodic orbit of $G \sharp F$.  Note that $(x,v)$ is also a capped orbit of the Hamiltonian $0\sharp F$, because $\supp (G) \subset M \setminus \Fix_c(\psi)$.  To prove the claim it suffices to show that $\Psi(x,v)$ is of the form
         \begin{equation*}
          \Psi(x,v) = a_{(x,v) }(x,v) + \sum a_{(y,w)} (y,w),  
          \end{equation*}
        with  $a_{(x,v)} \neq 0$ and $a_{(y,w)} = 0 $ if $\mathcal{A}_{0\sharp F}(y,w) \geq \mathcal{A}_{G\sharp F}(x,v)$ (and $(y,w) \neq (x,v)$). \textcolor{olive}{NEW:} Our argument for proving that $\Psi$ has the above form is similar to the proof of Corollary 4.11 in \cite{Bai-Xu}. 
    
       First, we explain why $a_{(x, v)} \neq 0$.  Since $\supp (G) \subset M \setminus \Fix_c(\psi)$, any interpolating homotopy is constant in a neighborhood of the common (contractible) one-periodic orbits of $X_{G\sharp F}$ and $X_{0\sharp F}$. Hence, for any $J^s$, the constant (in $s$) solution satisfies the parameterized Floer equation \eqref{eq:fe_cont} and so it constitutes an element of  $ \overline{\mathcal{M}((x,v), (x,v))}$. It follows that $\overline{\mathcal{M}\left((x,v), (x,v)\right)}$  consists of a single element, the constant solution; see Equation \eqref{eq:action_decreseas}. 
        Now, by \cite[Lem. 2.2]{Sa-lecture},  the constant solutions to \eqref{eq:fe_cont} are regular in the sense that the associated linearized Floer operator is surjective (here we use the fact that the interpolating homotopy is constant in a neighborhood of the orbits).  Hence, the moduli space $\mathcal{M}\left((x,v), (x,v)\right)$ is cut out transversely and therefore, the virtual count (which defines the continuation map $\Psi$) coincides with the geometric count of elements in $\overline{\mathcal{M}\left((x,v), (x,v)\right)}$ (see \cite[Lem.\ 5.2.6, Prop.\ 7.8.4]{Pardon}) which is $1$.  Hence, $a_{(x, v)} \neq 0$.

        To show that $a_{(y,w)} = 0 $ if $\mathcal{A}_{0\sharp F}(y,w) \geq \mathcal{A}_{G\sharp F}(x,v)$, we reason in a similar manner: we conclude from Equation \eqref{eq:action_decreseas} that $ \overline{\mathcal{M}((x,v), (y,w))} = \emptyset$. Therefore, the virtual and the geometric counts are both zero.  This completes the proof of the claim.
\end{proof}

Continuing with the proof of Lemma \ref{lem:spec}, we first observe the continuation map $\Psi$ is injective at the chain-level and hence a chain-level isomorphism since
    $$
 \dim_{\Lambda_{\Gamma}} CF(G \sharp F) = \dim_{\Lambda_{\Gamma}} CF(0 \sharp F).
    $$

    Knowing that $\Psi$ is a chain-level isomorphism, we conclude from the definition of the spectral invariant $c$ combined with the commutativity of PSS and continuations maps in homology that $$
    c( G \sharp F )=c(0\sharp F).
    $$

    It remains to explain the degenerate case. Fix $F$ with $\varphi_F=\psi$ and let $\tilde{F}$ be a non-degenerate perturbation of $F$ supported away from $\supp(G)$ and such that $\Fix_c(\varphi_{G} \varphi_{\Tilde{F}})=\Fix_c(\varphi_{\tilde{F}})$. By the previous part, we have $c( G \sharp \tilde{F} )=c(\tilde{F})$. Then, by Lipschitz continuity \eqref{eq:cont}, $c( G \sharp F )=c(F)$ holds as well. This completes the proof of the lemma.  
 
    \end{proof}

\begin{lemma}
\label{lem:conj}
    Let $\psi \in \Ham(M, \omega)$ and $G \colon \R/ \Z \times M \to \R$ be a non-negative Hamiltonian with $\supp (G) \subset M \setminus \Per_c(\psi)$. Suppose that $\Per_c(\varphi_{G} \psi)=\Per_c(\psi)$. Then, for all $k \in \N$ and for every Hamiltonian $F$ with $\varphi_F=\psi^k$, we have 
\[
c(F)= c (G\sharp G\circ \psi^{-1} \sharp \cdots \sharp G\circ\psi^{-(k-1)} \sharp F).
\]
\end{lemma}

\begin{proof}
    We will be applying Lemma \ref{lem:spec}.
    Fix $k \in \N$ and $F$ with $\varphi_F=\psi^k$. Set $\tilde{G}=G\sharp G \circ\psi^{-1} \sharp \cdots \sharp G\circ\psi^{-(k-1)} $. Note that
\[
\supp(\tilde{G}) = \bigcup_{i=0}^{k-1} \psi^{k} (\supp (G)).
\]
    Since $\supp (G) \subset M \setminus \Per_c(\psi)$ we have $\supp(\tilde{G}) \subset M \setminus \Fix_c (\psi^k)$ as well. 
    We next show that $\Fix_c(\varphi_{\tilde{G}} \psi^k)=\Fix_c(\psi^k)$.  Note that once this is established the result follows immediately from Lemma \ref{lem:spec}.
    
    We claim that the Hamiltonian diffeomorphism $\varphi_{\tilde{G}}$ has following form
\begin{equation}\label{eq:form_tildeG}
\varphi_{\tilde{G}} = \varphi_{G} \circ \left(\psi \varphi_{G} \psi^{-1} \right) \circ \left( \psi^2 \varphi_{G} \psi^{-2} \right) \circ \cdots  \circ \left(\psi^{k-1} \varphi_{G} \psi^{-(k-1)}\right).
\end{equation}
To see why the above is true, first note that, by Section \ref{sec:hamiltonians}, 
$$\varphi_{\tilde{G}} = \varphi_{G} \circ \varphi_{G\circ \psi^{-1}} \circ \cdots \circ  \varphi_{G\circ \psi^{-(k-1)}}.$$ 
Now recall from Section \ref{sec:hamiltonians} that  $\varphi_{G\circ \psi^{-i}} =\psi^{i} \varphi_{G}\psi^{-i}$; \eqref{eq:form_tildeG} follows immediately.  

As a consequence of \eqref{eq:form_tildeG} we have
\[
\varphi_{\tilde{G}} \psi^k = (\varphi_{G} \psi)^k
\]
and hence, since we are supposing that  $\Per_c(\varphi_{G} \psi)=\Per_c(\psi)$, we get $\Fix_c(\varphi_{\tilde{G}} \psi^k)=\Fix_c(\psi^k)$.  This completes the proof of the lemma. 
\end{proof}

\begin{theorem}
\label{thm:main}
    Let $\psi \in \Ham(M, \omega)$ and $G \colon \R/ \Z \times M \to \R$ be a non-negative Hamiltonian with $\supp (G) \subset M \setminus \Per_c(\psi)$. Suppose that $\Per_c(\varphi_{G} \psi)=\Per_c(\psi)$. Then, for all $k \in \N$ we have 
      $$c(G) \leq \gamma(\psi^k).$$
\end{theorem}

\begin{proof}
    Fix any $k \in \N$ and any $F$ satisfying $\varphi_F=\psi^k$.  We will show that $$c(G) \leq \gamma(F) .$$  This proves the theorem because, by definition, $\gamma(\psi^k) = \inf \{\gamma(F): \varphi_F = \psi^k \}.$ By Lemma \ref{lem:conj}, we have
\[
c(F)= c (G\sharp G\circ \psi^{-1} \sharp \cdots \sharp G \circ \psi^{-(k-1)} \sharp F).
\]
Next we use Triangle inequality \eqref{eq:tri} to write
\[
c(G\sharp G \circ \psi^{-1} \sharp \cdots \sharp G \circ \psi^{-(k-1)} \sharp F \sharp \bar{F}) \leq c(F) + c(\bar{F})=\gamma (F). 
\]
To finish the proof, it suffices to show that
\[
c(G) \leq c(G\sharp G \circ \psi^{-1} \sharp \cdots \sharp G \circ \psi^{-(k-1)} \sharp F \sharp \bar{F}).
\]

Using Homotopy invariance \eqref{eq:normal} we cancel $F\sharp \bar{F}$ from right hand side:
\[
c(G\sharp G\circ \psi^{-1} \sharp \cdots \sharp G\circ\psi^{-(k-1)} \sharp F \sharp \bar{F})=c(G\sharp G \circ \psi^{-1} \sharp \cdots \sharp G \circ \psi^{-(k-1)}). 
\]
This can be done because, as mentioned in Remark \ref{rk:conca}, the two Hamiltonians $G\sharp G \circ \psi^{-1} \sharp \cdots \sharp G \circ \psi^{-(k-1)} $ and $G\sharp G\circ \psi^{-1} \sharp \cdots \sharp G\circ\psi^{-(k-1)} \sharp F \sharp \bar{F}$ generate same element in  $\tHam$ and, moreover, $ \Delta(G\sharp G \circ \psi^{-1} \sharp \cdots \sharp G \circ \psi^{-(k-1)})=\Delta(G\sharp G\circ \psi^{-1} \sharp \cdots \sharp G\circ\psi^{-(k-1)} \sharp F \sharp \bar{F}).$

As the final step let us modify $G$. Set 
$G'(t,x)=kG(kt, x)$ for $t\in [0,1/k]$ and $G'(t,x)=0$ otherwise. By Homotopy invariance \eqref{eq:normal}, we have $c(G)= c(G')$ and 
\[
G'\leq G\sharp G \circ \psi^{-1} \sharp \cdots \sharp G\circ \psi^{-(k-1)}.
\]
For the latter we used the non-negativity of $G$. Using Monotonicity \eqref{eq:monotone} we conclude that
\[
c(G) \leq c(G\sharp G\circ \psi^{-1} \sharp \cdots \sharp G\circ \psi^{-(k-1)}).
\]
This completes the proof of the theorem.
\end{proof}

Next using Theorem \ref{thm:main} we prove Theorem \ref{thm:closing}. 

\begin{proof}[Proof of Theorem \ref{thm:closing}]
We first handle the non-negative case.  
Let $G \geq 0$ be a non-zero Hamiltonian with $\supp (G) \subset M \setminus \Per_c(\psi)$. For a contradiction, suppose that the composition $\varphi_{G} \psi$ has no periodic point passing through $\supp (G)$. By Theorem \ref{thm:main}, for all $k \in \N$, we have 
\[
c(G)\leq \gamma(\psi^k).
\]
     This implies that $c(G) \leq 0$ since $\psi$ is $\gamma$-rigid. Let us see that this is not possible.

     Let $H \geq 0$ be any non-negative Hamiltonian with $H \leq G$. Since $H \geq 0$, the inverse Hamiltonian $\bar{H} \leq 0$, see Section \ref{sec:hamiltonians}. We have, $\bar{H} \leq H \leq G$. Using Monotonicity \eqref{eq:monotone} we conclude that $c(\bar{H}) \leq c(H) \leq c(G) \leq 0$ and hence
\begin{equation}
\label{eq:zero}
    \gamma(H) =c(H) + c(\bar{H}) \leq 0. 
\end{equation}
    By the non-degeneracy of the $\gamma$-norm, see Section \ref{sec:spec}, the inequality  \eqref{eq:zero} implies that $\varphi_H = \Id$. This is a contradiction, since one can always find a Hamiltonian $0\leq H \leq G$ with $\varphi_H \neq \Id$.

Next we argue the non-positive case. Let $G \leq 0$ be a non-zero Hamiltonian. Observe that $\bar{G} \geq 0$ and $\supp (\bar{G}) = \supp (G)$ (see Section \ref{sec:hamiltonians}). Since the inverse map $\psi^{-1}$ is also $\gamma$-rigid, see Section \ref{sec:spec}, we conclude from the non-negative case that $\varphi_{\bar{G}} \psi^{-1}$ has a periodic point passing through $\supp (\bar{G})$.
In other words, there exists $x \in \supp (\bar{G})$ and $k \in \N$ such that  $(\varphi_{\bar{G}} \psi^{-1})^k(x)=x$. Set $y= (\varphi_{\bar{G}})^{-1} (x) \in \supp (\bar{G})$. We have $(\varphi_{\bar{G}})^{-1}(\varphi_{\bar{G}} \psi^{-1})^{-k} \varphi_{\bar{G}} (y) =y$  and
\begin{align*}
     (\varphi_{\bar{G}})^{-1}  (\varphi_{\bar{G}} \psi^{-1})^{-k} \varphi_{\bar{G}} = (\varphi_{\bar{G}})^{-1}  (\psi (\varphi_{\bar{G}})^{-1})^{k} \varphi_{\bar{G}}  
     = ((\varphi_{\bar{G}})^{-1} \psi )^k.
\end{align*}
Hence $y \in \supp(\bar{G})=\supp(G)$ is a $k$-periodic point of $(\varphi_{\bar{G}})^{-1} \psi =\varphi_{G} \psi$  (see Section \ref{sec:hamiltonians}). This completes the proof of the theorem. 
\end{proof}

\section{The strong closing property for ellipsoids}\label{sec:ellipsoids}
As mentioned in the introduction, it is possible to prove the strong closing property for ellipsoids from the strong closing property for rotations of $(\C P^n, \omega_{FS})$. Here, we briefly outline this.   

  As noted by Xue \cite{Xue}, one can recast the closing property for ellipsoids as a closing property for quadratic Hamiltonians in $\C^n$.  We begin by explaining this.    Let $E\subset \C^n$ be an ellipsoid; up to a linear  symplectomorphism it can be written in the form $E= Q^{-1}(1)$ where 
\begin{equation}\label{eq:quadratic-ellipsoid}
    Q(z_1, \ldots, z_n)= \sum_{i=1}^n a_i \vert z_i \vert^2.
\end{equation}
Consider the 1-form $\lambda = \frac{1}{2} \sum_{i=1}^n y_i dx_i -x_i dy_i$, where $x_i, y_i$ are the real and imaginary parts of $z_i$.  It induces a contact form on $E$ whose Reeb\footnote{The Reeb flow of $\alpha$ is the flow of the vector field $R$ satisfying the following two conditions : $d\alpha(R, \cdot) =0 $ and $\alpha(R) =1$. } flow coincides, up to reparametrization, with the Hamiltonian flow of $Q$ on $E$.  A small perturbation of the contact form induces a perturbed Reeb flow which can be seen as a perturbation of the Hamiltonian $Q$.  The strong closing property for ellipsoids, proven in \cite{CDPT}, states that for for every non-zero function $g: E \to \R_{\ge 0}$  (or $g: E \to \R_{\le 0}$) there exists $s\in[0, 1]$ such that the  contact form $(1+sg) \lambda$ (on $E$) has a closed Reeb orbit passing through the support of $g$.  For any choice of $g$ as above one can find a compactly supported function $f: \C^n \to \R$ such that the Reeb flow of $(1+sg) \lambda$ corresponds to the Hamiltonian flow on some level set of $Q+f$ and vice versa.  Hence, to deduce the strong closing property for ellipsoids, it is sufficient to prove  the following statement.   

\begin{proposition}\label{prop:reeb}
    Consider a compactly supported non-zero function $f: \C^n \rightarrow \R_{\ge 0}$ (or  $f: \C^n \rightarrow \R_{\le 0}$).  The Hamiltonian flow of $Q+f$ has a periodic orbit passing through the support of $f$.
\end{proposition}

Next we outline how the above proposition follows from Theorem \ref{thm:closing}.

We think of $(\C P^n, \omega_{FS})$ as the symplectic quotient of the unit sphere $ (S^{2n+1}, \omega_{n+1}  )  \subset \C^{n+1}$, where $\omega_{n+1}$ is the standard symplectic form on  $\C^{n+1}$, by the diagonal $S^1$ action 
 $$ t \cdot (z_0, z_1, \ldots, z_n) := (e^{2\pi i t} z_0, e^{2\pi i t} z_1, \ldots, e^{2\pi i t} z_n).$$

 Now, we view the quadratic Hamiltonian $Q$, defined by \eqref{eq:quadratic-ellipsoid}, as a Hamiltonian on $\C^{n+1}$, equipped with coordinates $(z_0, z_1, \ldots, z_n)$, which does not depend on $z_0$.  The restriction of $Q$ to the sphere $S^{2n+1}$ is invariant under the above circle action and so defines a (quadratic) Hamiltonian on $\C P^n$ which we denote by  $Q_{\C P^n}$.  The associated Hamiltonian diffeomorphism $\varphi_{Q_{\C P^n}}$ is a rotation of $(\C P^n, \omega_{FS})$.  Such maps  are $\gamma$-rigid (see\cite{Ginzburg-Gurel18a, Ginzburg-Gurel19}) and hence satisfy the strong closing property of Theorem \ref{thm:closing}.

There is a natural symplectic embedding $i: (B, \omega_n) \hookrightarrow  (\C P^n, \omega_{FS})$ of the unit ball 
$B \subset \C^n$ given by 
\[  
(z_1, \ldots, z_n) \mapsto \left[z_0:= \sqrt{1 - \sum \vert z_i \vert^2}: z_1: \ldots: z_n \right];
\]
see, for example, \cite{Karshon}. Observe that
\begin{equation*}
\label{eq:cpn-cn}
    Q_{\C P^n} \circ i = Q.
\end{equation*}

Up to rescaling the ellipsoid $E$, and multiplying $Q$ by a constant, we may suppose that $E = Q^{-1}(1)$ is contained in the interior of $B$.  Similarly, we may assume that $f$ is also supported in the interior of $B$.  Consequently, to prove Proposition \ref{prop:reeb}, it is sufficient to show that for any non-zero function $f: \C P^n \rightarrow \R_{\ge 0}$ (or  $f: \C P^n \rightarrow \R_{\le 0}$), the Hamiltonian flow of $Q_{\C P^n}+f$ has a periodic orbit passing through the support of $f$.

\begin{lemma}
\label{lemma:reeb}
Theorem \ref{thm:closing} implies that the Hamiltonian flow of ${Q_{\C P^n}+f}$ has a periodic orbit passing through the support of $f$.
\end{lemma}
Note that this is not an immediate consequence of Theorem \ref{thm:closing} because of difference between $\sharp$ and $+$.
\begin{proof}
    To simplicity our notation, throughout this proof we set $Q =Q_{\C P^n}$. 
    
    Let $G = \bar{Q}\# (Q+ f)$ where $H\#F(t, x):= H(t,x) + F(t, \varphi^{-t}_H(x))$; the flow of  $H\#F$ is given $\varphi^t_{H\#F}= \varphi^t_H \circ \varphi^t_F.$ We have
\begin{align*}
    G (t,x)  & = -Q(\varphi^t_Q(x)) + Q( \varphi^t_Q(x)) + f( \varphi^t_Q(x))\\
     & = f( \varphi^t_Q(x)) \ge 0.
\end{align*}

Observe that
\begin{equation} \label{eqn:support-G}
    \supp(G) =  \bigcup_{t=0}^1 \varphi^{-t}_{Q}(\supp (f)).
\end{equation}

Note that $\varphi_{Q+f} = \varphi_Q \varphi_G$ and so we can apply Theorem \ref{thm:closing} to conclude  that $\varphi_{Q+f}$ has  a periodic point passing through $\supp(G)$. Consequently the Hamiltonian flow of the $Q+f$  has a closed orbit passing through $\supp(G)$.

It remains to show that the closed orbit passes through the set $\supp(f)$.  This can be deduced from \eqref{eqn:support-G} as follows.  Any orbit  of $Q+f$ which does not pass through the support of $f$ is actually an orbit of $Q$.  By \eqref{eqn:support-G}, any closed orbit of $Q$ that passes through $\supp(G)$ must also pass through $\supp(f)$.
\end{proof}


 \bibliographystyle{alpha}
 \bibliography{biblio}

{\small

\medskip
 \noindent Erman \c C\. inel\. i\\
\noindent Sorbonne Universit\'e and Universit\'e de Paris, CNRS, IMJ-PRG, F-75006 Paris, France.\\
 {\it e-mail:} erman.cineli@imj-prg.fr

\medskip
 \noindent Sobhan Seyfaddini\\
\noindent Sorbonne Universit\'e and Universit\'e de Paris, CNRS, IMJ-PRG, F-75006 Paris, France.\\
 {\it e-mail:} sobhan.seyfaddini@imj-prg.fr
 
}

\end{document}